\documentclass[reqno]{amsart}
\usepackage{url}
\usepackage[dvips]{graphicx}
\newtheorem{thm}{Theorem}
\newtheorem{lem}{Lemma}          
\newtheorem{defn}{Definition}          
\begin{document}
\bibliographystyle{plain}

\AmS -\TeX
\title[Constructive proofs of Tychonoff's fixed point theorem]{Constructive proofs of Tychonoff's and Schauder's fixed point theorems for sequentially locally non-constant functions}

\author{Yasuhito Tanaka}
\address{Faculty of Economics, Doshisha University, Kamigyo-ku, Kyoto, 602-8580, Japan}
\email{yasuhito@mail.doshisha.ac.jp}

\thanks{This work was supported in part by the Ministry of Education, Science, Sports and Culture of Japan, Grant-in-Aid for Scientific Research (C), 20530165, and the Special Costs for Graduate Schools of the Special Expenses for Hitech Promotion by the Ministry of Education, Science, Sports and Culture of Japan in 2011.}


\begin{abstract}
We present a constructive proof of Tychonoff's fixed point theorem in a locally convex space for sequentially locally non-constant functions, As a corollary to this theorem we also present Schauder's fixed point theorem in a Banach space for sequentially locally non-constant functions. We follow the Bishop style constructive mathematics.
\end{abstract}

\maketitle

\section{Introduction}

It is well known that Brouwer's fixed point theorem can not be constructively proved in general case. Sperner's lemma which is used to prove Brouwer's theorem, however, can be constructively proved. Some authors, for example \cite{da} and \cite{veld}, have presented a constructive (or an approximate) version of Brouwer's theorem using Sperner's lemma. Thus, Brouwer's fixed point theorem can be constructively proved in its constructive version. Also Dalen in \cite{da} states a conjecture that a uniformly continuous function $f$ from a simplex to itself, with property that each open set contains a point $x$ such that $x\neq f(x)$ and also for every point $x$ on the faces of the simplex $x\neq f(x)$, has an exact fixed point. We call such a property \emph{local non-constancy}. Further we define a stronger property \emph{sequential local non-constancy}. In another paper \cite{ta1} using Sperner's lemma for modified partition of a simplex we have constructively proved Dalen's conjecture with sequential local non-constancy. 

In this paper, also using the modified version of Sperner's lemma, we will constructively present Tychonoff's fixed point theorem in a locally convex space and prove Schauder's fixed point theorem as a corollary to Tychonoff's theorem\footnote{Formulations of Tychonoff's and Schauder's fixed point theorems in this paper follow those in \cite{is}.}.

 In the next section we prove Sperner's lemma for modified partition of a simplex. In Section 3 we present Tychonoff's fixed point theorem and Schauder's fixed point theorem. We follow the Bishop style constructive mathematics according to \cite{bb}, \cite{br} and \cite{bv}.

\section{Sperner's lemma}\label{sec2}
To prove Sperner's lemma we use the following simple result of graph theory, Handshaking lemma\footnote{For another constructive proof of Sperner's lemma, see \cite{su}. }. A \emph{graph} refers to a collection of vertices and a collection of edges that connect pairs of vertices. Each graph may be undirected or directed. Figure \ref{graph} is an example of an undirected graph. Degree of a vertex of a graph is defined to be the number of edges incident to the vertex, with loops counted twice. Each vertex has odd degree or even degree. Let $v$ denote a vertex and $V$ denote the set of all vertices.
\begin{lem}[Handshaking lemma]
Every undirected graph contains an even number of vertices of odd degree. That is, the number of vertices that have an odd number of incident edges must be even.
\end{lem}
This is a simple lemma. But for completeness of arguments we provide a proof.
\begin{proof}
Prove this lemma by double counting. Let $d(v)$ be the degree of vertex $v$. The number of vertex-edge incidences in the graph may be counted in two different ways: by summing the degrees of the vertices, or by counting two incidences for every edge. Therefore,
\[\sum_{v\in V}d(v)=2e,\]
where $e$ is the number of edges in the graph. The sum of the degrees of the vertices is therefore an even number. It could happen if and only if an even number of the vertices had odd degree.

\end{proof}

\begin{figure}[t]
\begin{center}
\includegraphics[height=5cm]{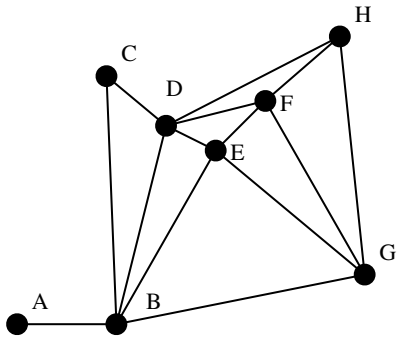}
\end{center}
	\vspace*{-.3cm}
	\caption{Example of graph}
	\label{graph}
\end{figure}

Let $\Delta$ denote an $n$-dimensional simplex. $n$ is a finite natural number. For example, a 2-dimensional simplex is a triangle.  Let partition or triangulate a simplex. Figure \ref{tria1} is an example of partition (triangulation) of a 2-dimensional simplex. In a 2-dimensional case we divide each side of $\Delta$ in $m$ equal segments, and draw the lines parallel to the sides of $\Delta$. Then, the 2-dimensional simplex is partitioned into $m^2$ triangles. We consider partition of $\Delta$ inductively for cases of higher dimension. In a 3 dimensional case each face of $\Delta$ is an 2-dimensional simplex, and so it is partitioned into $m^2$ triangles in the way above mentioned, and draw the planes parallel to the faces of $\Delta$. Then, the 3-dimensional simplex is partitioned into $m^3$ trigonal pyramids. And similarly for cases of higher dimension.

\begin{figure}[t]
\begin{center}
\includegraphics[height=7.5cm]{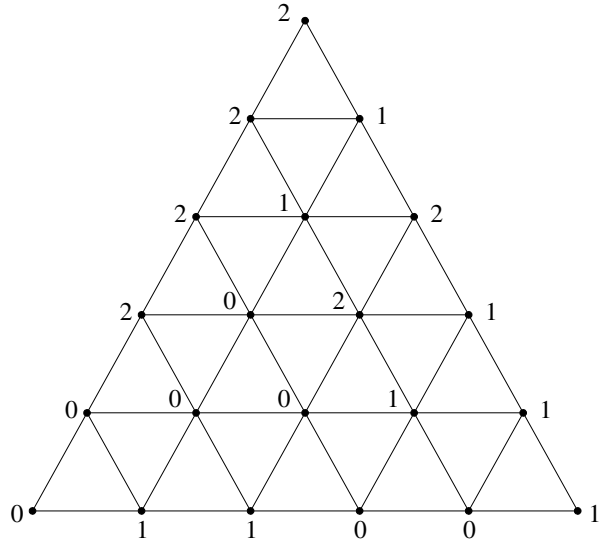}
\end{center}
	\vspace*{-.3cm}
	\caption{Partition and labeling of 2-dimensional simplex}
	\label{tria1}
\end{figure}

Let $K$ denote the set of small $n$-dimensional simplices of $\Delta$ constructed by partition. Vertices of these small simplices of $K$ are labeled with the numbers 0, 1, 2, $\dots$, $n$ subject to the following rules.
\begin{enumerate}
\item The vertices of $\Delta$ are respectively labeled with 0 to $n$. We label a point $(1,0, \dots, 0)$ with 0, a point $(0,1,0, \dots, 0)$ with 1, a point $(0,0,1 \dots, 0)$ with 2, $\dots$, a point $(0,\dots, 0,1)$ with $n$. That is, a vertex whose $k$-th coordinate ($k=0, 1, \dots, n$) is $1$ and all other coordinates are 0 is labeled with $k$. 

\item If a vertex of $K$ is contained in an $n-1$-dimensional face of $\Delta$, then this vertex is labeled with some number which is the same as the number of a vertex of that face.

\item If a vertex of $K$ is contained in an $n-2$-dimensional face of $\Delta$, then this vertex is labeled with some number which is the same as the number of a vertex of that face. And similarly for cases of lower dimension.

\item A vertex contained inside of $\Delta$ is labeled with an arbitrary number among 0, 1, $\dots$, $n$.
\end{enumerate}

Now we modify this partition of a simplex as follows.
\begin{quote}
Put a point in an open neighborhood around each vertex inside $\Delta$, and make partition of $\Delta$ replacing each vertex inside $\Delta$ by that point in each neighborhood. The diameter of each neighborhood should be sufficiently small relatively to the size of each small simplex. We label the points in $\Delta$ following the rules (1) $\sim$ (4).
\end{quote}
Then, we obtain a partition of $\Delta$ illustrated in Figure \ref{tria12}.

We further modify this partition as follows;
\begin{quote}
Put a point in an open neighborhood around each vertex on a face (boundary) of $\Delta$, and make partition of $\Delta$ replacing each vertex on the face by that point in each neighborhood, and we label the points in $\Delta$ following the rules (1) $\sim$ (4). This neighborhood is open in a space with dimension lower than $n$.
\end{quote}
Then, we obtain a partition of $\Delta$ depicted in Figure \ref{tria13}.% in a 2-dimensional case.

\begin{figure}[t]
\begin{center}
\includegraphics[height=7.5cm]{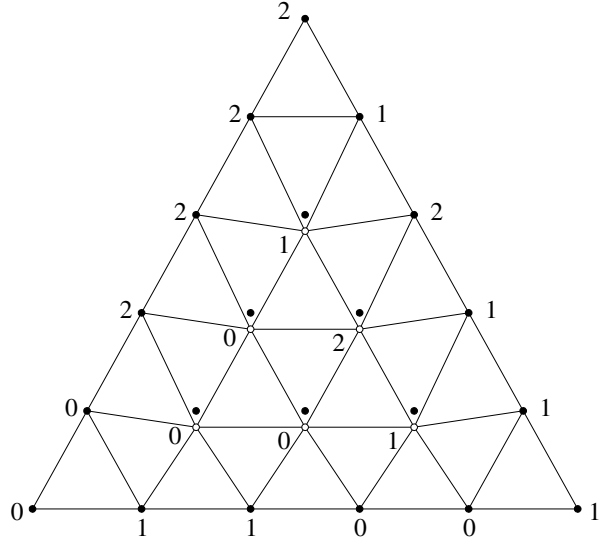}
\end{center}
	\vspace*{-.3cm}
	\caption{Modified partition of a simplex}
	\label{tria12}
\end{figure}

A small simplex of $K$ in this modified partition which is labeled with the numbers 0, 1, $\dots$, $n$ is called a \emph{fully labeled simplex}. Now let us prove Sperner's lemma about the modified partition of a simplex.
\begin{lem}[Sperner's lemma]
If we label the vertices of $K$ following above rules (1) $\sim$ (4), then there are an odd number of fully labeled simplices. Thus, there exists at least one fully labeled simplex. \label{l2}
\end{lem}
\begin{proof}
See Appendix \ref{app1}.
\end{proof}

\begin{figure}[t]
\begin{center}
\includegraphics[height=8cm]{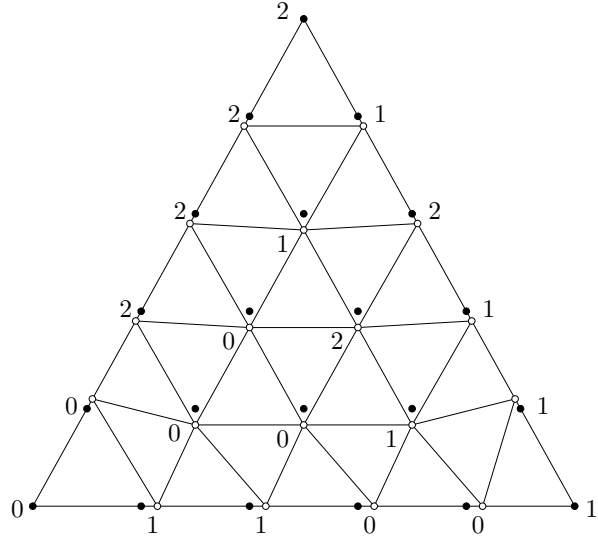}
\end{center}
	\vspace*{-.3cm}
	\caption{Modified partition of a simplex: Two}
	\label{tria13}
\end{figure}

\section{Tychonoff's and Schauder's fixed point theorems for sequentially locally non-constant functions}

In this section we prove Tychonoff's fixed point theorem for sequentially locally non-constant functions in a locally convex space using Sperner's lemma and present Schauder's fixed theorem as a corollary to Tychonoff's fixed point theorem. Our Tychonoff's fixed point theorem is stated as follows;

\begin{thm}[Tychonoff's fixed point theorem for sequentially locally non-constant and uniformly continuous functions]
Let $X$ be a compact (totally bounded and complete) and convex subset of a locally convex space $E$, and $g$ be a sequentially locally non-constant and uniformly continuous function from $X$ to itself. Then, $g$ has a fixed point.
\end{thm}

A locally convex space consists of a vector space $E$ and a family $(p_i)_{i\in I}$ of seminorms on $X$. $I$ is an index set, for example, the set of positive integers. According to \cite{bv} we define, constructively, total boundedness of a set in a locally convex space as follows;
\begin{defn}[Total boundedness of a set in a locally convex space]
Let $X$ be a subset of $E$, $F$ be a finitely enumerable subset of $I$\footnote{A set $S$ is finitely enumerable if there exist a natural number $N$ and a mapping of the set $\{1, 2, \dots, N\}$ onto $S$.}, and $\varepsilon>0$. By an $\varepsilon$-approximation to $X$ relative to $F$ we mean a subset $T$ of $X$ such that for each $x\in X$ there exists $y\in T$ with $\sum_{i\in F}p_i(x-y)<\varepsilon$. $X$ is totally bounded relative to $F$ if for each $\varepsilon>0$ there exists a finitely enumerable $\varepsilon$-approximation to $X$ relative to $F$. It is totally bounded if it is totally bounded relative to each finitely enumerable subset of $I$.\label{tb}
\end{defn}

And uniform continuity of a function in a locally convex space is defined as follows;

\begin{defn}[Uniform continuity of a function in a locally convex space]
Let $X$, $Y$ be subsets in a locally convex space. A function $g:\ X\longrightarrow Y$ is uniformly continuous in $X$ if for each $\varepsilon>0$ and each finitely enumerable subset $G$ of $J$, which is also an index set, there exist $\delta>0$ and a finitely enumerable subset $F$ of $I$ such that if $x, y\in X$ and $\sum_{i\in F}p_i(x-y)<\delta$, then $\sum_{j\in G}q_j(g(x)-g(y))<\varepsilon$, where $(q_j)_{j\in J}$ is a family of seminorms on $Y$.\label{uc}
\end{defn}
In a metric space or Banach space, a seminorm should be replaced by a metric or a norm in these definitions.

Since $X$ is totally bounded, there exists a finitely enumerable $\varepsilon$-approximation $\{x^0, x^1, \dots, x^n\}$ to $X$. Consider an $n$-dimensional simplex $\Delta$ in Euclidean space with vertices $v^0=(1, 0, 0, \dots, 0)$, $v^1=(0, 1, 0, \dots, 0)$, $\dots$, $v^n=(0, 0, \dots, 1)$. Consider a point $v\in \Delta$ such that $v=\sum_{j=0}^n\alpha_jv^j$, and a function $h$ such that $h: \Delta \longrightarrow X$ and  $h(v)=\sum_{j=0}^n\alpha_jx^j$, where $\sum_{j=0}^n\alpha_j=1,\ \alpha_j\geq 0,\ j=0, 1, \dots, n$. $h$ is clearly a uniformly continuous function. $h(v^j)=x^j$, $h^{-1}(x^j)=v^j$ for all $j\in \{0, 1, \dots, n\}$, and
\[h^{-1}(x)=\sum_{j=0}^n\alpha_jv^j\]
for $x=\sum_{j=0}^n\alpha_jx^j$. $h^{-1}$ is also uniformly continuous. Uniform continuity of $h$ and $h^{-1}$ is described as follows;
\begin{quote}
$h$ is uniformly continuous in $\Delta$ if for each $\varepsilon>0$ and each finitely enumerable subset $F$ of $I$ there exists $\delta>0$ such that if $v, u\in \Delta$ and $|v-u|<\delta$, then $\sum_{i\in F}p_i(h(v)-h(u))<\varepsilon$. 

$h^{-1}$ is uniformly continuous in $X$ if for each $\varepsilon>0$ there exists $\delta>0$ and a finitely enumerable subset $F$ of $I$ such that if $x, y\in X$ and $\sum_{i\in F}(x-y)<\delta$, then $|h^{-1}(x)-h^{-1}(y)|<\varepsilon$.
\end{quote}

Consider a function in a locally convex space $g: X\longrightarrow X$. Then, we can construct a function $f:\Delta \longrightarrow \Delta$ such that
\[f(v)=h^{-1}\circ g\circ h(v).\]

We define local non-constancy and modified local non-constancy of functions in an $n$-dimensional simplex $\Delta$ as follows;
\begin{defn}[Local non-constancy of functions]
\begin{enumerate}
	\item At a point $v$ on a boundary of $\Delta$, $f(v)\neq v$. This means $f_i(v)>v_i$ or $f_i(v)<v_i$ for at least one $i$, where $f_i(v)$ and $v_i$ are the $i$-th components of $f(v)$ and $v$. We use similar notation for other variables.
	\item In any open set of $\Delta$ there exists a point $v$ such that $f(v)\neq v$.
\end{enumerate}
\end{defn}
\begin{defn}[Modified local non-constancy of functions]
\begin{enumerate}
\item At the vertices of $\Delta$, $f(v)\neq v$.
\item In any open set contained in the faces (boundaries) of $\Delta$ there exists a point $v$ such that $f(v)\neq v$. This open set is open in a space of dimension lower than $n$.
\item In any open set of $\Delta$ there exists a point $v$ such that $f(v)\neq v$.
\end{enumerate}
\end{defn}
(2) of the modified local non-constancy implies that every vertex $v$ in a partition of a simplex, for example, as illustrated by white circles in Figure \ref{tria13} in a 2-dimensional case, can be selected to satisfy $f(v)\neq v$ even when points on the faces of $\Delta$ (black circles on the edges) do not necessarily satisfy this condition. Even if a function $f$ does not strictly satisfy the local non-constancy so long as it satisfies the modified local non-constancy, we can partition $\Delta$ to satisfy the conditions for Sperner's lemma.

Let $\partial X^i$ be a set of points in $X$ to which points in each $i$-dimensional face of an $n$-dimensional simplex $\Delta$ correspond by $h$ for $i=0, 1, \dots, n-1$. Since each face of $\Delta$ is totally bounded and $h$ is uniformly continuous, $\partial X^i$ is totally bounded for each $i$.

We define local non-constancy and modified local non-constancy of functions in a locally convex space as follows;
\begin{defn}[Local non-constancy of functions in a locally convex space]
\begin{enumerate}
	\item At a point $x$ to which a point on a boundary of $\Delta$ corresponds by $h$, $\sum_{i\in F}p_i(g(x)-x)>0$.% This means $g_i(x)>x_i$ or $g_i(x)<x_i$ for at least one $i$. 
	\item In any open set of $X$ there exists a point $x$ such that $\sum_{i\in F}p_i(g(x)-x)>0$, for each finitely enumerable subset $F$ of $I$ (in the same way hereafter).
\end{enumerate}
\end{defn}
\begin{defn}[Modified local non-constancy of functions in a locally convex space]
\begin{enumerate}
\item At each point $x$ to which each vertex of $\Delta$ corresponds by $h$, $\sum_{i\in F}p_i(g(x)-x)>0$.
\item In any open set in $\partial X^i$ for each $i$ there exists a point $x$ such that $\sum_{i\in F}p_i(g(x)-x)>0$.% This open set is open in a space of dimension lower than $n$.
\item In any open set of $X$ there exists a point $x$ such that $\sum_{i\in F}p_i(g(x)-x)>0$.
\end{enumerate}
\end{defn}

If $g$ satisfies the modified local non-constancy, $f$ also satisfies the modified local non-constancy. 

Next, by reference to the notion of \emph{sequentially at most one maximum} in \cite{berg}, we define the property of \emph{sequential local non-constancy} for $f:\ \Delta \longrightarrow \Delta$. Each face (boundary) of $\Delta$ is also a simplex, and so it is compact in a space with dimension lower than $n$. The definition of sequential local non-constancy is as follow;
\begin{defn}[Sequential local non-constancy of functions]
\begin{enumerate}
\item At the vertices of a simplex $\Delta$ $f(v)\neq v$.
\item There exists $\bar{\varepsilon}>0$ with the following property. We have a finitely enumerable $\varepsilon$-approximation $L=\{v^1, v^2, \dots, v^l\}$ to each face of $\Delta$ for each $\varepsilon$ with $0<\varepsilon<\bar{\varepsilon}$ such that if for all sequences $(v(m))_{m\geq 1}$, $(u(m))_{m\geq 1}$ in each open $\varepsilon$-ball $S'$, which is a subset of the face, around each $v^i\in L$ $|f(v(m))-v(m)|\longrightarrow 0$ and $|f(u(m))-u(m)|\longrightarrow 0$, then $|v(m)-u(m)|\longrightarrow 0$. $S'$ is open in a space with dimension lower than $n$.
\item For $\bar{\varepsilon}$ defined above there exists a finitely enumerable  $\varepsilon$-approximation $L=\{v^1, v^2, \dots, v^l\}$ to $\Delta$ for each $\varepsilon$ with $0<\varepsilon<\bar{\varepsilon}$ such that if for all sequences $(v(m))_{m\geq 1}$, $(u(m))_{m\geq 1}$ in each open $\varepsilon$-ball $S'$ in $\Delta$ around each $v^i\in L$ $|f(v(m))-v(m)|\longrightarrow 0$ and $|f(u(m))-u(m)|\longrightarrow 0$, then $|v(m)-u(m)|\longrightarrow 0$.

\end{enumerate}\label{sln}
\end{defn}

Similarly, we define sequential local non-constancy for functions in a locally convex space as follows;
\begin{defn}[Sequential local non-constancy of functions in a locally convex space]
\begin{enumerate}
\item At each point $x$ to which each vertex of $\Delta$ corresponds by $h$, $\sum_{i\in F}p_i(g(x)-x)>0$.

\item There exists $\bar{\varepsilon}>0$ with the following property. We have a finitely enumerable  $\varepsilon$-approximation $L=\{x^1, x^2, \dots, x^l\}$ to $\partial X^i$ for each $i$ and each $\varepsilon$ with $0<\varepsilon<\bar{\varepsilon}$ such that if for all sequences $(x(m))_{m\geq 1}$, $(y(m))_{m\geq 1}$ in each open $\varepsilon$-ball $S'$, which is a subset of $\partial X^i$, around each $x^j\in L$ $\sum_{i\in F}p_i(g(x(m))-x(m))\longrightarrow 0$ and $\sum_{i\in F}p_i(g(y(m))-y(m))\longrightarrow 0$, then $\sum_{i\in F}p_i(x(m)-y(m))\longrightarrow 0$.% $S'$ is open in a space with dimension lower than $n$.

\item For $\bar{\varepsilon}$ defined above there exists a finitely enumerable $\varepsilon$-approximation $L=\{x^1, x^2, \dots, x^l\}$ to $X$ for each $\varepsilon$ with $0<\varepsilon<\bar{\varepsilon}$ such that if for all sequences $(x(m))_{m\geq 1}$, $(y(m))_{m\geq 1}$ in each open $\varepsilon$-ball $S'$ in $X$ around each $x^j\in L$ $\sum_{i\in F}p_i(g(x(m))-x(m))\longrightarrow 0$ and $\sum_{i\in F}p_i(g(y(m))-y(m))\longrightarrow 0$, then $\sum_{i\in F}p_i(x(m)-y(m))\longrightarrow 0$.

\end{enumerate}\label{sln2}
\end{defn}
%(1) of this definition is the same as (1) of the definition of modified local non-constancy. 
(1) of this definition is the same as (1) of the definition of modified local non-constancy. 

If $g$ satisfies the sequential local non-constancy, $f$ also satisfies the sequential local non-constancy. 

Now we show the following two lemmas.
\begin{lem}
Sequential local non-constancy means modified local non-constancy in a locally non-convex space.\label{seq}
\end{lem}
The essence of this proof is according to the proof of Proposition 1 of \cite{berg}.
\begin{proof}
Let $S'$ be a set as defined in (2) or (3) of Definition \ref{sln2}. Construct a sequence $(z(m))_{m\geq 1}$ in $S'$ such that $\sum_{i\in F}p_i(g(z(m))-z(m))\longrightarrow 0$. Consider $x$, $y$ in $S'$ with $\sum_{i\in F}p_i(x-y)>0$. Construct an increasing binary sequence $(\lambda_m)_{m\geq 1}$ such that
\[\lambda_m=0\Rightarrow \max\left(\sum_{i\in F}p_i(g(x)-x), \sum_{i\in F}p_i(g(y)-y)\right)<2^{-m},\]
\[\lambda_m=1\Rightarrow \max\left(\sum_{i\in F}p_i(g(x)-x), \sum_{i\in F}p_i(g(y)-y)\right)>2^{-m-1}.\]
We may assume that $\lambda_1=0$. If $\lambda_m=0$, set $x(m)=x$ and $y(m)=y$. If $\lambda_m=1$, set $x(m)=y(m)=z(m)$. Now the sequences $(\sum_{i\in F}p_i(g(x(m))-x(m)))_{m\geq 1}$, $(\sum_{i\in F}p_i(g(y(m))-y(m)))_{m\geq 1}$ converge to 0, and so by sequential local non-constancy $\sum_{i\in F}p_i(x(m)-y(m))\longrightarrow 0$. Computing $M$ such that $\sum_{i\in F}p_i(x(M)-y(M)<\sum_{i\in F}p_i(x-y)$, we see that $\lambda_M=1$. Therefore, $\sum_{i\in F}p_i(g(x)-x)>0$ or $\sum_{i\in F}p_i(g(y)-y)>0$.

Let $S$ be an open set in $X$. Then there exists an $\varepsilon$-approximation to $X$ with sufficiently small $\varepsilon$ such that an $\varepsilon$-ball around some point in that $\varepsilon$-approximation is included in $S$. Therefore, there exists a point $x$ in $S$ such that $\sum_{i\in F}p_i(g(x)-x)>0$.
\end{proof}

\begin{lem}
Let $g$ be a uniformly continuous function from a compact and convex set $X$ to itself in a locally convex space, and assume that $\inf_{x\in S}\sum_{i\in F}p_i(g(x)-x)=0$ where $S$ is nonempty, compact and $S\subset X$. If the following property holds:
\begin{quote}
For each $\varepsilon>0$ there exists $\eta>0$ such that if $x, y\in S$, $\sum_{i\in F}p_i(g(x)-x)<\eta$ and $\sum_{i\in F}p_i(g(y)-y)<\eta$, then $\sum_{i\in F}p_i(x-y)\leq \varepsilon$.
\end{quote}
Then, there exists a point $\xi\in X$ such that $g(\xi)=\xi$, that is, a fixed point of $g$. \label{fix0}
\end{lem}
\begin{proof}
Choose a sequence $(x(m))_{m\geq 1}$ in $S$ such that $\sum_{i\in F}p_i(g(x(m))-x(m))\longrightarrow 0$. Compute $M$ such that $\sum_{i\in F}p_i(g(x(m))-x(m))<\eta$ for all $m\geq M$. Then, for $l, m\geq M$ we have $\sum_{i\in F}p_i(x(l)-x(m))\leq \varepsilon$. Since $\varepsilon>0$ is arbitrary, $(x(m))_{m\geq 1}$ is a Cauchy sequence in S, and converges to a limit $\xi\in S$. The continuity of $g$ yields $\sum_{i\in F}p_i(g(\xi)-\xi)=0$, that is, $g(\xi)=\xi$.
\end{proof}
In a metric space or Banach space a seminorm should be replaced by a metric or norm in these lemmas.

Let us prove Tychonoff's fixed point theorem (Theorem 1).

\begin{proof}
We prove this theorem through some steps.
\begin{enumerate}
\item First we show that we can partition $\Delta$ so that the conditions for Sperner's lemma (for modified partition of a simplex) are satisfied. We partition $\Delta$ according to the method in the proof of Sperner's lemma, and label the vertices of simplices constructed by partition of $\Delta$. It is important how to label the vertices contained in the faces of $\Delta$. Let $K$ be the set of small simplices constructed by partition of $\Delta$, $v=(v_0, v_1, \dots, v_n)$ be a vertex of a simplex of $K$, and denote the $i$-th coordinate of $f(v)$ by $f_i$. We label a vertex $v$ according to the following rule,
\[\mathrm{If}\ v_k>f_k,\ \mathrm{we\ label}\ v\ \mathrm{with}\ k.\]
If there are multiple $k$'s which satisfy this condition, we label $v$ conveniently for the conditions for Sperner's lemma to be satisfied.% We do not randomly label the vertices.

Let us check labeling for vertices in three cases.
\begin{enumerate}
	\item Vertices of $\Delta$:

One of the coordinates of a vertex $v$ of $\Delta$ is 1, and all other coordinates are zero. Consider a vertex $(1, 0, \dots, 0)$. By the modified local non-constancy $f(v)\neq v$ means $f_j>v_j$ or $f_j<v_j$ for at least one $j$. $f_i<v_i$ can not hold for $i\neq 0$. On the other hand, $f_0>v_0$ can not hold. When $f_0(v)<v_0$, we label $v$ with 0. Assume that $f_i(v)>v_i=0$ for some $i\neq 0$. Then, since $\sum_{j=0}^nf_j(v)=1=v_0$, we have $f_0(v)<v_0$. Therefore, $v$ is labeled with 0. Similarly a vertex $v$ whose $k$-th coordinate is 1 is labeled with $k$ for all $k\in \{0, 1, \dots, n\}$.

	\item Vertices in the faces of $\Delta$:

Let $v$ be a vertex of a simplex contained in an $n-1$-dimensional face of $\Delta$ such that $v_i=0$ for one $i$ among $0, 1, 2, \dots, n$ (its $i$-th coordinate is 0). $f(v)\neq v$ means that $f_j>v_j$ or $f_j<v_j$ for at least one $j$. $f_i<v_i=0$ can not hold. When $f_k<v_k$ for some $k\neq i$, we label $v$ with $k$. Assume $f_i>v_i=0$. Then, since $\sum_{j=0}^nv_j=\sum_{j=0}^nf_j=1$, we have $f_k<v_k$ for some $k\neq i$, and we label $v$ with $k$. Assume that $f_j>v_j$ for some $j\neq i$. Then, since $v_i=0$ and
\[\sum_{l=0, l\neq j}^n f_l<\sum_{l=0, l\neq j}^nv_l,\]
we have $f_k<v_k$ for some $k\neq i, j$, and we label $v$ with $k$.

We have proved that we can label each vertex of a simplex contained in an $n-1$-dimensional face of $\Delta$ such that $v_i=0$ for one $i$ among $0, 1, 2, \dots, n$ with a number other than $i$. By similar procedures we can show that we can label the vertices of a simplex contained in an $n-2$-dimensional face of $\Delta$ such that $v_i=0$ for two $i$'s among $0, 1, 2, \dots, n$ with a number other than those $i$'s, and so on.

\begin{quote}
Consider a case where, for example, $v_{i}=v_{i+1}=0$. Neither $f_i<v_i=0$ nor $f_i<v_{i+1}=0$ can hold. When $f_k<v_k$ for some $j\neq i, i+1$, we label $v$ with $k$. Assume $f_i>v_i=0$ or $f_{i+1}>v_{i+1}=0$. Then, since $\sum_{j=0}^nv_j=\sum_{j=0}^nf_j=1$, we have $f_k<v_k$ for some $k\neq i,\ i+1$, and we label $v$ with $k$. Assume that $f_j>v_j$ for some $j\neq i, i+1$. Then, since $v_i=v_{i+1}=0$ and
\[\sum_{l=0, l\neq j}^n f_l<\sum_{l=0, l\neq j}^nv_l,\]
we have $f_k<v_k$ for some $k\neq i, i+1, j$, and we label $v$ with $k$.
\end{quote}

	\item Vertices of small simplices inside $\Delta$:

By the modified local non-constancy of $f$ every vertex $v$ in a modified partition of a simplex can be selected to satisfy $f(v)\neq v$. Assume that $f_i>v_i$ for some $i$. Then, since $\sum_{j=0}^n v_j=\sum_{j=0}^n f_j=1$, we have
\[f_k<v_k\]
for some $k\neq i$, and we label $v$ with $k$.
\end{enumerate}

Therefore, the conditions for Sperner's lemma (for modified partition of a simplex) are satisfied, and there exist an odd number of fully labeled simplices in $K$.

\item Suppose that we partition $\Delta$ sufficiently fine so that the distance between any pair of the vertices of simplices of $K$ is sufficiently small. Let $\delta^n$ be a fully labeled $n$-dimensional simplex of $K$, and $v^0, v^1, \dots$ and $v^n$ be the vertices of $\delta^n$. We name these vertices so that $v^0, v^1, \dots, v^n$ are labeled, respectively, with 0, 1, $\dots$, $n$. The values of $f$ at theses vertices are $f(v^0), f(v^1), \dots$ and $f(v^n)$. The $j$-th coordinates of $v^i$ and $f({v^i}),\ i=0, 1, \dots, n$ are, respectively, denoted by $v^i_j$ and $f_j(v^i)$. About $v^0$, from the labeling rules we have $v^0_0>f_0(v^0)$. About $v^1$, also from the labeling rules $v^1_1>f_1(v^1)$. Since $n$ is finite, by the uniform continuity of $f$ there exists $\eta>0$ such that if $|v^i-v^j|<\eta$, then $|f(v^i)-f(v^j)|<\frac{\varepsilon}{2n(n+1)}$ for $\varepsilon>0$ and $i\neq j$. $|f(v^0)-f(v^1)|<\frac{\varepsilon}{2n(n+1)}$ means $f_1(v^1)>f_1(v^0)-\frac{\varepsilon}{2n(n+1)}$. On the other hand, $|v^0-v^1|<\eta$ means $v^0_1>v^1_1-\eta$. We can make $\eta$ satisfying $\eta<\frac{\varepsilon}{2n(n+1)}$. Thus, from
\[v^0_1>v^1_1-\eta,\ v^1_1>f_1(v^1),\ f_1(v^1)>f_1(v^0)-\frac{\varepsilon}{2n(n+1)}\]
we obtain
\[v^0_1>f_1(v^0)-\eta-\frac{\varepsilon}{2n(n+1)}>f_1(v^0)-\frac{\varepsilon}{n(n+1)}\]
By similar arguments, for each $i$ other than 0,
\begin{equation}
v^0_i>f_i(v^0)-\frac{\varepsilon}{n(n+1)}. \label{e1}
\end{equation}
For $i=0$ we have
\begin{equation}
v^0_0>f_0(v^0). \label{e2}
\end{equation}
Adding (\ref{e1}) and (\ref{e2}) side by side except for some $i$ (denote it by $k$) other than 0,
\[\sum_{j=0, j\neq k}^{n} v^0_j>\sum_{j=0, j\neq k}^{n} f_j(v^0)-\frac{(n-1)\varepsilon}{n(n+1)}.\]
From $\sum_{j=0}^{n} v^0_j=1$, $\sum_{j=0}^{n} f_j(v^0)=1$ we have $1-v^0_k>1-f_k(v^0)-\frac{(n-1)\varepsilon}{n(n+1)}$, which is rewritten as
\[v^0_k<f_k(v^0)+\frac{(n-1)\varepsilon}{n(n+1)}.\]
Since (\ref{e1}) implies $v^0_k>f_k(v^0)-\frac{\varepsilon}{n(n+1)}$, we have
\[f_k(v^0)-\frac{\varepsilon}{n(n+1)}<v^0_k<f_k(v^0)+\frac{(n-1)\varepsilon}{n(n+1)}.\]
Thus,
\begin{equation}
|v^0_k-f_k(v^0)|<\frac{(n-1)\varepsilon}{n(n+1)}. \label{e3}
\end{equation}
On the other hand, adding (\ref{e1}) from 1 to $n$ yields
\begin{equation*}
\sum_{j=1}^{n} v^0_j>\sum_{j=1}^{n} f_j(v^0)-\frac{\varepsilon}{(n+1)}.%\label{nash1}
\end{equation*}
From $\sum_{j=0}^{n} v^0_j=1$, $\sum_{j=0}^{n} f_j(v^0)=1$ we have
\begin{equation}
1-v^0_0>1-f_0(v^0)-\frac{\varepsilon}{(n+1)}.\label{nash1}
\end{equation}
Then, from (\ref{e2}) and (\ref{nash1}) we get
\begin{equation}
|v^0_0-f_0(v^0)|<\frac{\varepsilon}{(n+1)}. \label{e24}
\end{equation}
From (\ref{e3}) and (\ref{e24}) we obtain the following result,
\begin{equation*}
|v^0_i-f_i(v^0)|<\frac{\varepsilon}{(n+1)}\ \mathrm{for\ all}\ i. %\label{fp}
\end{equation*}
Thus,
\begin{equation}
|v^0-f(v^0)|<\varepsilon.\label{fp}
\end{equation}
Since $\varepsilon$ is arbitrary, $\inf_{v\in \delta^n}|f(v)-v|=0$.% for $S\in \Delta.

\item Choose a sequence $(\xi(m))_{m\geq 1}$ in $\delta^n$ such that $|f(\xi(m))-\xi(m)|\longrightarrow 0$. In view of Lemma \ref{fix0} it is enough to prove that the following condition holds.
\begin{quote}
For each $\varepsilon>0$ there exists $\eta>0$ such that if $v, u\in \delta^n$, $|f(v)-v|<\eta$ and $|f(u)-u|<\eta$, then $|v-u|\leq \varepsilon$.
\end{quote}
Assume that the set
\[K=\{(v,u)\in \delta^n\times \delta^n:\ |v-u|\geq \varepsilon\}\]
is nonempty and compact. Since the mapping $(v,u)\longrightarrow \max(|f(v)-v|,|f(u)-u|)$ is uniformly continuous, we can construct an increasing binary sequence $(\lambda(m))_{m\geq 1}$ such that
\[\lambda_m=0\Rightarrow \inf_{(v,u)\in K}\max(|f(v)-v|,|f(u)-u|)<2^{-m},\]
\[\lambda_m=1\Rightarrow \inf_{(v,u)\in K}\max(|f(v)-v|,|f(u)-u|)>2^{-m-1}.\]
It suffices to find $m$ such that $\lambda_m=1$. In that case, if $|f(v)-v|<2^{-m-1}$, $|f(u)-u|<2^{-m-1}$, we have $(v,u)\notin K$ and $|v-u|\leq \varepsilon$. Assume $\lambda_1=0$. If $\lambda_m=0$, choose $(v(m), u(m))\in K$ such that $\max(|f(v(m))-v(m)|, |f(u(m))-u(m)|)<2^{-m}$, and if $\lambda_m=1$, set $v(m)=u(m)=\xi(m)$. Then, $|f(v(m))-v(m)|\longrightarrow 0$ and $|f(u(m))-u(m)|\longrightarrow 0$, so $|v(m)-u(m)|\longrightarrow 0$. Computing $M$ such that $|v(M)-u(M)|<\varepsilon$, we must have $\lambda_M=1$. Note that $f$ is a sequentially locally non-constant uniformly continuous function from $\Delta$ to itself. Thus, $f$ has a fixed point.

\item Let $v^*$ be a fixed point of $f$ and $x^*=h(v^*)$. Then, from
\[f(v^*)=h^{-1}\circ g\circ h(v^*),\]
we have
\[v^*=h^{-1}\circ g(x^*),\]
and
\[x^*=h(v^*)=g(x^*).\]
Therefore, $x^*$ is a fixed point of $g$.
\end{enumerate}

We have completed the proof.
\end{proof}

A Banach space is a locally convex space. Therefore, as a corollary to the constructive version of Tychonoff's fixed point theorem we obtain the following theorem.

\begin{thm}[Schauder's fixed point theorem for sequentially locally non-constant and uniformly continuous functions]
Let $X$ be a compact (totally bounded and complete) and convex subset of a Banach space $E$, and $g$ be a sequentially locally non-constant and uniformly continuous function from $X$ to itself. Then, $g$ has a fixed point.
\end{thm}

\appendix
\section{Proof of Sperner's lemma}\label{app1}
We prove Sperner's lemma by induction about the dimension of $\Delta$. When $n=0$, we have only one point with the number 0. It is the unique 0-dimensional simplex. Therefore the lemma is trivial. When $n=1$, a partitioned 1-dimensional simplex is a segmented line. The endpoints of the line are labeled distinctly, by 0 and 1. Hence in moving from endpoint 0 to endpoint 1 the labeling must switch an odd number of times, that is, an odd number of edges labeled with 0 and 1 may be located in this way.

\begin{figure}[t]
\begin{center}
\includegraphics{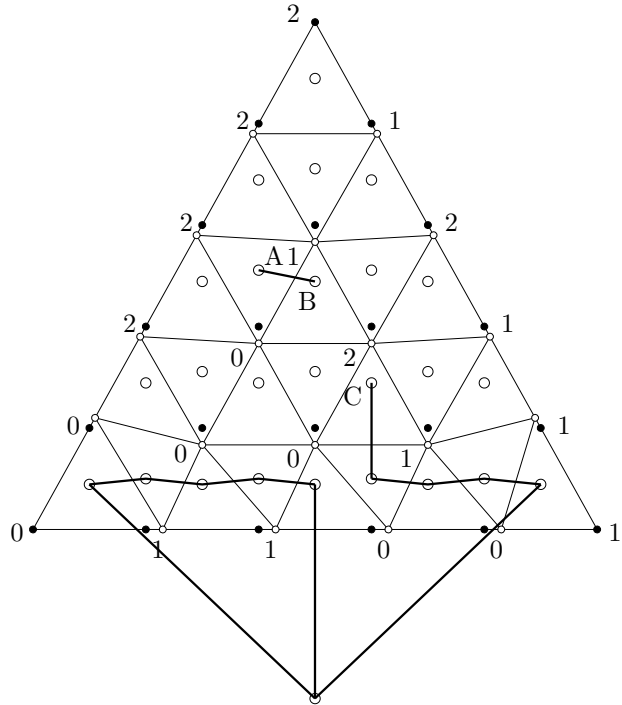}
\end{center}
	\vspace*{-.3cm}
	\caption{Sperner's lemma}
	\label{tria2}
\end{figure}

Next consider the case of 2 dimension. Assume that we have partitioned a 2-dimensional simplex (triangle) $\Delta$ as explained above. Consider the face of $\Delta$ labeled with 0 and 1\footnote{We call edges of triangle $\Delta$ \emph{faces} to distinguish between them and edges of a dual graph which we will consider later.}. It is the base of the triangle in Figure \ref{tria2}. Now we introduce a dual graph that has its nodes in each small triangle of $K$ plus one extra node outside the face of $\Delta$ labeled with 0 and 1 (putting a dot in each small triangle, and one dot outside $\Delta$). We define edges of the graph that connect two nodes if they share a side labeled with 0 and 1. See Figure \ref{tria2}. White circles are nodes of the graph, and thick lines are its edges. Since from the result of 1-dimensional case there are an odd number of faces of $K$ labeled with 0 and 1 contained in the face of $\Delta$ labeled with 0 and 1, there are an odd number of edges which connect the outside node and inside nodes. Thus, the outside node has odd degree. Since by the Handshaking lemma there are an even number of nodes which have odd degree, we have at least one node inside the triangle which has odd degree. Each node of our graph except for the outside node is contained in one of small triangles of $K$. Therefore, if a small triangle of $K$ has one face labeled with 0 and 1, the degree of the node in that triangle is 1: if a small triangle of $K$ has two such faces, the degree of the node in that triangle is 2, and if a small triangle of $K$ has no such face, the degree of the node in that triangle is 0. Thus, if the degree of a node is odd, it must be 1, and then the small triangle which contains this node is labeled with 0, 1 and 2 (fully labeled). In Figure \ref{tria2} triangles which contain one of the nodes $A$, $B$, $C$ are fully labeled triangles.

Now assume that the theorem holds for dimensions up to $n-1$. Assume that we have partitioned an $n$-dimensional simplex $\Delta$. Consider the fully labeled face of $\Delta$ which is a fully labeled $n-1$-dimensional simplex. Again we introduce a dual graph that has its nodes in small $n$-dimensional simplices of $K$ plus one extra node outside the fully labeled face of $\Delta$ (putting a dot in each small $n$-dimensional simplex, and one dot outside $\Delta$). We define the edges of the graph that connect two nodes if they share a face labeled with 0, 1, $\dots$, $n-1$. Since from the result of $n-1$-dimensional case there are an odd number of fully labeled faces of small simplices of $K$ contained in the $n-1$-dimensional fully labeled face of $\Delta$, there are an odd number of edges which connect the outside node and inside nodes. Thus, the outside node has odd degree. Since, by the Handshaking lemma there are an even number of nodes which have odd degree, we have at least one node inside the simplex which has odd degree. Each node of our graph except for the outside node are contained in one of small $n$-dimensional simplices of $K$. Therefore, if a small simplex of $K$ has one fully labeled face, the degree of the node in that simplex is 1: if a small simplex of $K$ has two such faces, the degree of the node in that simplex is 2, and if a small simplex of $K$ has no such face, the degree of the node in that simplex is 0. Thus, if the degree of a node is odd, it must be 1, and then the small simplex which contains this node is fully labeled.

\begin{quote}
If the number (label) of a vertex other than vertices labeled with 0, 1, $\dots$, $n-1$ of an $n$-dimensional simplex which contains a fully labeled $n-1$-dimensional face is $n$, then this $n$-dimensional simplex has one such face, and this simplex is a fully labeled $n$-dimensional simplex. On the other hand, if the number of that vertex is other than $n$, then the $n$-dimensional simplex has two such faces.
\end{quote}

We have completed the proof of Sperner's lemma.

Since $n$ and partition of $\Delta$ are finite, the number of small simplices constructed by partition is also finite. Thus, we can constructively find a fully labeled $n$-dimensional simplex of $K$ through finite steps.

\bibliography{yasuhito}
\end{document}